\documentclass[12pt,a4paper,oneside]{article}
\usepackage{amsfonts,amsmath,amsthm,amssymb,amscd}
\usepackage{epsfig}%, fullpage}

\usepackage{verbatim}
\newcommand{\beq}{\begin{equation}}
\newcommand{\eeq}{\end{equation}}
\newcommand{\bea}{\begin{eqnarray}}
\newcommand{\eea}{\end{eqnarray}}
\newcommand{\ba}{\begin{array}}
\newcommand{\ea}{\end{array}}

\usepackage{amsmath,amsfonts,amssymb,amsthm}
\usepackage{xspace}
\usepackage{color}

\newtheorem{thrm}{Theorem}[section]

\newtheorem{defn}{Definition}[section]

\newtheorem{example}{Example}

\begin{document}

\title{Tchebychev Polynomial Approximations for $m^{th}$ Order Boundary Value Problems}
\author{ Samir Kumar Bhowmik\\
 Department of Mathematics and Statistics,  College of Science\\ Al Imam Mohammad Ibn Saud Islamic University(IMSIU)\\
 P.O. BOX 90950, 11623 Riyadh, Kingdom of Saudi Arabia,  \\ e-mail: bhowmiksk@gmail.com
}
\maketitle
%{\bf ABSTRACT}
%
\begin{abstract}
Higher order boundary value problems (BVPs) play an important role modeling various scientific and engineering problems.
In this article we develop an efficient numerical scheme for linear $m^{th}$ order BVPs. First we convert the higher order BVP to a first order BVP. Then we use Tchebychev orthogonal polynomials to approximate the solution of the BVP as a weighted  sum of polynomials. We collocate at Tchebychev clustered grid points to generate a system of equations to approximate the weights for the polynomials. The excellency of the numerical scheme is illustrated through some examples.
\end{abstract}
\textbf{}
\section{Introduction}
Boundary value problems are an interesting topic in the field of applied mathematics, physics and engineering~\cite{Ahmet2008, R.P.Agarwa, Taha2012, DohaBS11, El-GamelCZ04, ShepleyLRoss}. Higher order BVPs arise in modeling various physical, chemical and biological realities. As for example, the  Schrödinger equation can be presented by a second order BVP~\cite{WeidemanR2000};  the free vibration analysis of beam structures is governed by a $4^{th}$ order differential equation; and that of ring structures by  a $6^{th}$ order differential equation; an ordinary convection yields a $10^{th}$ order BVP; an ordinary overstability yields a $12^{th}$ order BVP~\cite{Y.Wanga2006}.
Also an historically important example of a $4^{th}$ order BVP is the Orr Sommerfeld equation from the field of hydrodynamic stability~\cite{Trefethen}.  Numerical approximations for these problems are challenging because of the higher order derivatives and boundary conditions involving higher order derivatives of the unknown function.

In this study we consider the $m^{th}$ order non-homogeneous linear differential equation
\begin{equation}\label{f:eq001}
   y^{(m)}(t) + a_1(t)y^{(m-1)}(t) + a_2(t)y^{(m-2)}(t)+\cdots+ a_{m-1}(t)y^{(1)}(t) + a_m(t) y(t)= f(t),\ t \in (a,\ b),
\end{equation}
with a set of $m$ boundary conditions. Here $a_i(t)$, $i=1,\ 2,\cdots, m$ and $f(t)$ are real valued functions.  Here we consider
\eqref{f:eq001} with two exemplary set of boundary conditions though the scheme is not limited to them.
\begin{enumerate}
\item
\begin{enumerate}
\item[i.] $y^{(2k)}(a)=\alpha_{2k}$, $k=0,\ 1,\ 2,\ \cdots, m/2-1$,
\item[ii.] $y^{(2k)}(b)=\beta_{2k}$, $k=0,\ 1,\ 2,\ \cdots, m/2-1$, $m$ even,
\end{enumerate}
\item
\begin{enumerate}
\item[i.] $y^{(2k)}(a)=\alpha_{2k}$, $k=0,\ 1,\ 2,\ \cdots, (m-1)/2$
\item[ii.] $y^{(2k)}(b)=\beta_{2k}$, $k=0,\ 1,\ 2,\ \cdots, (m-1)/2-1$, $m$ odd.
\end{enumerate}
\end{enumerate}
or
\begin{enumerate}
\item
\begin{enumerate}
\item[i.] $y^{(k)}(a)=\alpha_{k}$, $k=0,\ 1,\ 2,\ \cdots, m/2-1$,
\item[ii.] $y^{(k)}(b)=\beta_{k}$, $k=0,\ 1,\ 2,\ \cdots, m/2-1$, $m$ even,
\end{enumerate}
\item
\begin{enumerate}
\item[i.] $y^{(k)}(a)=\alpha_{k}$, $k=0,\ 1,\ 2,\ \cdots, (m-1)/2$
\item[ii.] $y^{(k)}(b)=\beta_{k}$, $k=0,\ 1,\ 2,\ \cdots, (m-1)/2-1$, $m$ odd.
\end{enumerate}
\end{enumerate}

The existence and uniqueness of analytic solutions  of the higher order BVP \eqref{f:eq001} have been well studied~\cite{R.P.Agarwa, Xinan2014, raisinghania2007} and many references therein.
The numerical approximation  on the solution of the $m^{th}$-order
BVP \eqref{f:eq001}  is sparse.
So we are not interested in discussing the  existence and uniqueness issues rather we focus on an efficient numerical scheme for \eqref{f:eq001}.

% Some problems are contained implicitly in the work of several authors.
 The authors concentrate on numerical methods for $4^{th}$ order BVPs in~\cite{Chawla1979, Twizell.1988}.  A low-order numerical method is outlined in \cite{Twizell.1988}. Taher et. al~\cite{Taher2013} study $4^{th}$ order Sturm-Liouville BVPs. They use Tchebychev differentiation matrices to approximate the model. They present some examples to demonstrate the scheme.

A detailed study about a $6^{th}$ order BVPs can be found in \cite{N.P.HALE.2006}.  They update various existing MATLAB codes to solve the $6^{th}$ order differential equation with various boundary conditions using finite difference based schemes. A good discussion about the existence and uniqueness of solutions of such BVPs can be found in this article.

A locally supported Lagrange polynomial approximation of higher order BVPs can be found in \cite{Y.Wanga2006}. They approximate a $6^{th}$ order differential equation and an $8^{th}$ order differential equation, respectively to illustrate the proposed method with two different set of boundary conditions.

Recently Adomian decomposition methods for $4^{th}$ order
BVPs are used by  \cite{Attili2006}. Shi and Cao \cite{ZCao2012} approximates a higher order
eigenvalue problem using Haar wavelets. Whereas \cite{U.Ycel2012} applies a differential quadrature method for a $4^{th}$ order BVP.

The famous and popular texts on higher order BVPs using spectral method are \cite{Trefethen, WeidemanR2000}. In both of the study the authors consider Tchebychev nodes to approximate the higher order derivatives in the MATLAB programming environment. They present numerical  solutions of some $2^{nd}$ and $4^{th}$  order BVPs and eigenvalue problems.
Both the authors use some differentiation matrices to approximate the higher order derivatives. Then they discuss issues regarding boundary conditions.  They discuss some difficulties of handling boundary conditions, specially boundary conditions related to the higher order derivatives. The limitation of the scheme is that  nonuniqueness of approximate solutions may occur~ \cite{WeidemanR2000} because of the boundary conditions.

From all the discussions above and the references therein we experience that the main difficulty of using these schemes is to implement/handle boundary conditions. Studying all these articles we notice that most techniques are developed focusing some particular boundary conditions. Most of the articles are restricted to $4^{th}$ order problems. It is also noticed that some articles discuss schemes for even order BVPs and some for odd order BVPs.  Thus a complete study about a general higher order BVPs  is missing. Also, in most cases, the authors compute Tchebychef differentiation matrices (TDM) to approximate all higher order derivatives. The implementation of the boundary conditions in  TDM  is complicated~\cite{Trefethen, WeidemanR2000}.  Here we aim to develop a scheme where we do not need to compute TDMs, and handling boundary conditions is easy.

Here we propose a  simple and efficient scheme for a general higher order BVP \eqref{f:eq001}. This scheme works well for both the even and the odd order differential equations with any type of boundary conditions. The efficient recipe is based on the Tchebychev collocation  method. In this scheme one can handle boundary conditions without any complication. Actually we convert the higher order BVP to a first order BVP.  We approximate the resulting first order system of differential equations using orthogonal polynomials defined in a bounded interval $[a,\ b]$. Here all boundary conditions involving higher order derivatives have been converted to boundary conditions for the respective transformed functions. Then we present  the unknown functions and the boundary values by weighted sums of truncated Tchebychev polynomials. As a result the scheme gives us a system of linear equations for the Tchebychev weights/coefficients which can be solved by using a standard linear algebra solver.
We do not have to compute Tchebychev differentiation matrices for higher order derivatives here.
Thus the scheme becomes simple and efficient. The proposed scheme preserves spectral accuracy which has many applicabilities in physical and engineering models. Here along with the solution of \eqref{f:eq001} the proposed scheme generates an approximation of the first $m-1$ derivatives of the unknown function which is important in the study of some physical realities.

The rest of the article is organised as follows:  In Section~\ref{sec02} we convert the higher order differential equation  \eqref{f:eq001} to a set of first order differential equations.  We discuss some basic results on Tchebychev polynomials and recurrence relations to support our study in Section~\ref{sec03}.  A simple scheme for the resulting system of first order differential equations has been developed in Section~\ref{sec04}. We finish with some numerical examples and discussion in Section~\ref{sec05}.

\section{Reduction of order for the BVP \eqref{f:eq001}}\label{sec02}
Any higher order differential equation can be converted into a system of first order differential  equations. Thus a higher order BVP can be converted into a first order system of differential equations with the same boundary conditions.
Here in this study we aim to develop a scheme to approximate higher order BVPs after converting them to first order system of differential equations.
To that end using the transformations
\begin{equation}\label{trnsfrm01:f}
y=y_1,\ \frac{dy}{dt} =y_2,\ \frac{d^2y}{dt^2} =y_3,\ \cdots,\ \frac{d^{m-1}y}{dt^{m-1}} =y_m,
\end{equation}
we rewrite the $m^{th}$ order linear BVPs \eqref{f:eq001} as a system of ordinary differential equations
\begin{equation}\label{f:eq002}
\begin{gathered}
\frac{dy_1}{dx} =y_2,\ \frac{dy_2}{dx} =y_3,\ \cdots, \ \frac{dy_{m-1}}{dx} =y_m \\
\frac{dy_m}{dt} = -\sum_{i=1}^{m-1}a_i y_{m-i} + f(t).
\end{gathered}
\end{equation}
The system of equations \eqref{f:eq002} can be written as
\begin{equation}\label{f:eq002aa}
\frac{d\underline y}{dt} = A \underline y +\underline f(t),\ \quad t\in (a,\ b),
\end{equation}
where
\[
A=\left(
  \begin{array}{ccccccc}
  0 & 1 &0 &0&\cdots&0&0\\
  0&0 &1 & 0&0&\cdots&0\\
  \vdots & \ddots&\ddots & \ddots&\ddots&\ddots&\vdots\\
  0&0 &0 &\cdots&0&1&0\\
  0&0&\cdots&0&0&0&1\\
  -a_m(t)& -a_{m-1}(t)& \cdots& \cdots & -a_2(t)& -a_1(t)& 0
  \end{array}
   \right),
\]
\[
\underline y(t)=\left(
  \begin{array}{c}
 y_1(t)\\
 y_2(t)\\
 \vdots\\
 y_m(t)
  \end{array}
   \right),
\quad
\underline f(t)=\left(
  \begin{array}{c}
 0\\
0\\
 \vdots\\
f(t)
  \end{array}
   \right).
\]
The boundary condition for the $k^{th}$ derivative of $y(t)$ is now transformed into the boundary condition of the unknown function $y_k(t)$ for all $k=0, 1,\ 2,\ \cdots, m-1$.
So in the similar fashion the transformed boundary conditions can be written as
\begin{equation}\label{f:eq002_bc_a}
\begin{gathered}
\text{for }\ m\ \text{even }\left\{
\begin{array}{l}
y_{2k+1}(a)=\alpha_{2k},\ k=0,\ 1,\ 2,\ \cdots, m/2-1,\\
y_{2k+1}(b)=\beta_{2k},\ k=0,\ 1,\ 2,\ \cdots, m/2-1, \\
\end{array}
\right.
%\end{gathered}
\\
%\begin{gathered}
\text{for }\ m\ \text{odd }\left\{
\begin{array}{l}
y_{2k+1}(a)=\alpha_{2k}, k=0,\ 1,\ 2,\ \cdots, (m-1)/2,\\
y_{2k+1}(b)=\beta_{2k}, k=0,\ 1,\ 2,\ \cdots, (m-1)/2-1,
\end{array}
\right.
\end{gathered}
\end{equation}
or
\begin{equation}\label{f:eq002_bc_b}
\begin{gathered}
\text{for }\ m\ \text{even }\left\{
\begin{array}{l}
y_{k+1}(a)=\alpha_{k},\ k=0,\ 1,\ 2,\ \cdots, m/2-1,\\
y_{k+1}(b)=\beta_{k}, k=0,\ 1,\ 2,\ \cdots, m/2-1,\\
\end{array}
\right.
\\
\text{for }\ m\ \text{odd }\left\{
\begin{array}{l}
y_{k+1}(a)=\alpha_{k}, k=0,\ 1,\ 2,\ \cdots, (m-1)/2,\\
y_{k+1}(b)=\beta_{k}, k=0,\ 1,\ 2,\ \cdots, (m-1)/2-1.\\
\end{array}
\right.
\end{gathered}
\end{equation}
From now on we consider the system of first order differential equations \eqref{f:eq002aa} with the set of transforms boundary conditions \eqref{f:eq002_bc_a} and \eqref{f:eq002_bc_b}.
\section{Tchebychev polynomial  approximation}\label{sec03}
Before moving onto our main study we discuss some basic properties to support our  scheme. In this article we are concerned with approximating solutions  of \eqref{f:eq002aa} using Tchebychev polynomials as trials functions and Tchebychev nodes for collocation. Polynomial interpolations based on Tchebychev nodes are often used to approximate smooth function. The pseudospectral methods perform well, in cases where both solutions and coefficients are not smooth~\cite{B.Forn, W.Gautschi} as well.
In general, Spectral methods \cite{Trefethen} are a class of spatial discretizations for differential equations. They can be categorised as Galerkin, tau and collocation(or pseudospectral) spectral methods. Galerkin and tau work with the coefficients of a global
expansion whereas pseudospectral work with the values at collocation points. There are two key components for the formulation of Spectral methods:
\begin{itemize}
  \item Trial function, which is also called the expansion or approximation functions.
  \item Test function, which is also known as weight functions.
\end{itemize}

Let $L_n u(t)$ be the $n^{th}$ degree polynomial approximation of $u(t)$. We state an upper bound for the solutions generated by pseudospectral approximation.
%The fundamental problem of approximation of any function $u(t)$ is to find conditions under which $L_n u(t)$ converges to $u(t)$ as $n$ goes $\infty.$
%
\begin{defn}
An approximation scheme converges to the function approximated if
\[
 \lim_{n\rightarrow\infty}\|L_n u(t)-u(t)\|=0
\]
 for all $t\in\Omega,$ where $\Omega$ is the domain for $u(t)$, and
\begin{equation}\label{apprerr01:f}
L_n u(t) = \sum_{j=0}^{n} a_{j}\phi_{j} (t),
\end{equation}
  where $\phi_j(t)$ are orthogonal polynomials defined in  $\Omega$.
\end{defn}
\begin{thrm}~\cite{J.Shohat}
Let $u(t)$ be finite at every point of the finite interval $(a, b)$ and
such that $\int_{a}^{b}u^2 (t)dt$ exists. Then
\[
  \lim_{n\rightarrow \infty} \int_{a}^{b} |u(t)-L_n(u)|^r d\psi=0,
\quad \mbox{$0<r\le 2,$}
\]
if  $\{\phi_n (t)\}$ represents any set of
orthogonal polynomials corresponding to $(a, b),$ and $\psi$ is a weight function.
\end{thrm}
\begin{proof}
For exact details see~\cite{J.Shohat}.
\end{proof}

Now a days, a lot of attention has been grown to the study of the Tchebychev orthogonal polynomial approximation focusing various  real life models. The efficiency of the method is very important, and have been  well studied by several authors~\cite{Trefethen, WeidemanR2000}. The goal of this section is to recall some properties of the Tchebychev polynomials, state some known results, and derive useful formulas that are important for this study.

The Tchebychev points are unequally spaced points over $[-1,\ 1]$. These points are the horizontal coordinates of a unit circle with center $(0, 0)$.  It is to note that they are numbered from right to left. These points are the extreme points of the nth degree Tchebychev polynomial of the first kind. These points are defined by
\[
   t_i = \cos\left(\frac{i \pi}{n} \right),\  i=0,\ 1,\ 2,\cdots,\ n.
\]
Tchebychev polynomials can be defined over  $[-1,\ 1]$ and are obtained by expanding the following generating formula
\[
  T_r(t) = \cos(r \cos^{-1}(t)), \ r= 0,\ 1,\ 2,\cdots, \ -1\le t \le 1.
\]
This polynomials satisfy the following properties:
\begin{enumerate}
\item  $T_0(t)=1$, $T_1(t)=t$ and $T_r(t)=2tT_{r-1}(t)- T_{r-2}(t)$, $r\ge 2$.
\item
\[
\int_{-1}^1 \frac{T_r T_p}{\sqrt{1-t^2}} dt = \left\{
                       \begin{array}{cc}
                        0, & r\ne p\\
                        \frac{\pi}{2}, & r>0\\
                        \pi, & r=0.
                     \end{array}
                     \right.
\]
\end{enumerate}
 Any analytic function $u(t)$  can be approximated by a truncated Tchebychev series as
\[
 u(t) = \sum_{r=0}^n c_{r} T_r(t), -1\le t\le 1,
\]
which can be written as
\[
  u(t) = T(t) c,
\]
where $$T(t) = [T_0(t),\ T_1(t),\cdots, T_n(t)],$$ and $$c = [c_0,\ c_1,\cdots, c_n]'.$$
Here, following the above convergence result we see that if the function $u$ belongs to $C^{\infty}$ class, the produced error of
approximation as $n$ tends to infinity, approaches zero with exponential rate ($\mathcal{O}(e^{-c^+ n})$, for some $c^+ > 0$).
This phenomenon is usually referred to as ``spectral accuracy''.

Assuming the function is differentiable the derivatives of $u(t)$ can be written as~\cite{Msezer}
\begin{equation}\label{f:eq0054a}
 u^{(k)}(t) = \sum_{j=0}^n c_{j}^{(k)}T_r(t), -1\le t\le 1, k \ge 0.
\end{equation}

It is well known that the coefficients $c_{r}^{(k)}$ and $c_{r}^{(k+1)}$ of $u^{(k)}(t)$ and $u^{(k+1)}(t)$ can be written by  the following three term recurrence relation
\[
 2rc_{r}^{(k)} = c_{r-1}^{(k+1)}- c_{r+1}^{(k+1)},\ r\ge 1, %\ i=1,\ 2,\ 3, \cdots, M,
\]
can be simplified into the following relation~\cite{Msezer}
\[
c_{r}^{(k+1)} = 2 \sum_{j=0}^\infty (r+2j+1)c_{r+2i+1}^{(k)},
\]
where $c_{r}^{(0)} = c_{r}$.
Accordingly,
the truncated system $c^{(k)}$ can be computed as
\begin{equation}\label{f:eq0054b}
 c^{(k)}  = 2^k \mathcal{M}^k c,
\end{equation}
where
\[ \text{when} ~ n ~\text{is}~ \text{odd},~~~
\mathcal{M}  = \left(
                \begin{array}{ccccccc}
                 0 & \frac{1}{2} & 0 & \frac{3}{2}& 0 &\cdots & \frac{n}{2}\\
                 0 & 0 & 2& 0& 4&0 \cdots & 0\\
                 0 & 0 & 0& 3 & 0 & 5\cdots & n\\
                 \vdots&\vdots&\vdots&\vdots &&&\vdots\\
                 0& 0 & 0 & & \cdots && n\\
                 0 & 0 & 0 & 0 & 0 \cdots&0 & 0
                \end{array}
                \right),
 \]
 and
 \[ \text{when}~ n ~\text{is}~ \text{even},~~~
\mathcal{M}  = \left(
                \begin{array}{ccccccc}
                 0 & \frac{1}{2} & 0 & \frac{3}{2}& 0 &\cdots & 0\\
                 0 & 0 & 2& 0& 4&0 \cdots & n\\
                 0 & 0 & 0& 3 & 0 & 5\cdots & 0\\
                 0& 0 & 0 & & \cdots && n\\
                 0 & 0 & 0 & 0 & 0 \cdots&0 & 0
                \end{array}
                \right).
 \]
 \section*{Tchebychev polynomials on $[a, b]$}
Now the Tchebychev polynomials can be defined over an interval $[a,\ b]$ by
\begin{equation}\label{f:chebyab}
  T_r^*(t) = \cos(r \cos^{-1}\left(\frac{2}{b-a}\left(t-\frac{a+b}{2}\right)\right), \ r= 0,\ 1,\ 2,\cdots, \ a\le t \le b.
\end{equation}
The collocation points in $[a,\ b]$ can be defined by
\begin{equation}\label{nodes001:f}
 t_j = \frac{a-b}{2}\left(\frac{a+b}{b-a} + \cos \left(\frac{j\pi}{n} \right)  \right),\ j=0,\ 1,\ 2,\cdots, n.
\end{equation}
Also using the transformed Tchebychev polynomials a unknown function $v^{*}(t)$  and the derivatives  can be approximated by
\[
 v^{*(k)}(t) = \sum_{j=0}^n C_{j}^{*(k)} T_j^*(t),\ t \in [a,\ b],\ k=0,\ 1,\ 2, \cdots,
\]
where
\begin{equation}\label{f:eq0054c}
   C^{*(k)} = \left(\frac{4}{b-a}\right)^k \mathcal{M}^k C^*,
\end{equation}
%$$ T^*(t)  = \left[T^*_0(t)\ T^*_1(t)\ \cdots\ T^*_n(t)\right]', $$
 $$ C^* = [C_{0}^*\ C_{1}^*\ C_{2}^* \cdots C_{n}^*]'.$$%~;~i = 1,\ 2, \cdots, m,$$

\section{Polynomial approximation of the first order system of BVPs}\label{sec04}
In this section we motivate ourselves to solve the  system of differential equations \eqref{f:eq002aa} with a given set of  boundary conditions.
To that end we recall \eqref{f:eq002aa} with solutions as a truncated series of Tchebychef polynomials given by
 \begin{equation}\label{f:system001a}
   y_i(t) = \sum_{j=0}^n C_{i,j} T^*_j(t), a \le t \le b,\ i=1,\ 2,\ 3,\cdots, m
 \end{equation}
 which  can be expressed as
$$ y_i(t) =  T^*(t) C_i  ,$$ where $$ T^*(t)  = \left[T^*_0(t)\ T^*_1(t)\ \cdots\ T^*_n(t)\right], $$
and $$ C_i = [C_{i, 0}\ C_{i, 1}\ C_{i, 2} \cdots C_{i, n}]'~;~i = 1,\ 2, \cdots, m,$$ and $T^*_j(t)$ are given by \eqref{f:chebyab} in $[a\ b]$.  It is easy to see that $y_i(t)$ approximates  $y^{(i-1)}(t)$ for $i=1,\ 2,\ 3,\cdots, m$.
 Here we aim to compute $C_{i, j}$ values so that \eqref{f:system001a} approximates \eqref{f:eq002aa} satisfying all the boundary conditions.
 We write the derivatives of the unknown function $y_i(t)$ as
\begin{equation}\label{f:fff}
     y_i^{(k)}(t) = T^*(t) C_i^{(k)}.
\end{equation}
Combining \eqref{f:system001a}, \eqref{f:fff} and \eqref{f:eq0054c} the derivatives involved in \eqref{f:eq002aa}  can be expressed as
 \[
   y_i^{(k)}(t) = \left(\frac{4}{b-a}\right)^k T^*(t)\mathcal{M}^k C_i,\  k=0,\ 1.
 \]
Thus $\underline y(t) $ and all the derivatives in \eqref{f:eq002aa} can be approximated by
 \[
   \underline y^{(k)}(t) = \left(\frac{4}{b-a}\right)^k {  \Psi(t)} \mathcal{\mathbf{M}}^k C,\ k=0,\ 1,
 \]
% { \color{magenta} is this $T(t)$ or it should be $\Psi(t)$ }
%where
where
 \[
   \Psi(t)  = \left(
   \begin{array}{ccccc}
   T^*(t) &&&&\\
   &T^*(t) &&&\\
   &&\ddots &&\\
   &&&& T^*(t)
   \end{array}
   \right)_{m \times m},
 \]

 \[
   C = \left(
   \begin{array}{c}
   C_1 \\
   C_2 \\
   \vdots \\
    C_{m}
   \end{array}
   \right)_{m\times 1},
 \]
 and
 \[
  \mathcal{\mathbf{M}}^k  =
   \left(
   \begin{array}{ccccc}
   \mathcal{M}^k &&&&\\
   &\mathcal{M}^k &&&\\
   &&\ddots &&\\
   &&&& \mathcal{M}^k
   \end{array}
   \right)_{m\times m}.
  %diag\left( \mathcal{M}^k\ \mathcal{M}^k\cdots \mathcal{M}^k \right), \mbox{a $R(N+1)\times R(N+1)$}.
 \]

Using the above notations the system of ODEs \eqref{f:eq002aa} gives us a system of algebraic equations for the Tchebychef weights $C$ as
 \begin{equation}\label{flappr02:f}
U^{(1)}(t)-A U^{(0)}(t) = \underline f(t),
 \end{equation}
 where
 \[
   U^{(0)}(t) = \Psi(t) C,\
 \]
 and
 \[
    U^{(1)}(t) = \frac{4}{b-a} \Psi(t)\mathcal{\mathbf{M}} C.
 \]
To solve the system for $C_{i, j}$ we collocate \eqref{flappr02:f} at translated Tchebychev nodes \eqref{nodes001:f}.
Collocating at the prescribed grid points  \eqref{flappr02:f} yields
 \begin{equation}\label{f:full_discrete_01}
  P_1 \underline U^{(1)} -  P_0 \underline U^{(0)} = \underline F,
 \end{equation}
 where
 \[
  \underline U^{(i)},\ i=0,\ 1, \ \mbox{are $ m(n+1)$ column vectors,}
  %=  [u^{(i)}(t_0) \ u^{(i)}(t_1)\cdots \ u^{(i)}(t_N)]',\ \mbox{a $R(N+1)$ vector,}
 \]
 and
 \[
  \underline F =  [\underline f(t_0) \ \underline f(t_1)\cdots \ \underline f(t_n)]', \ \mbox{a $m(n+1)$ column vector,}
 \]
 \[
 P_1=\left(
     \begin{array}{ccccc}
    I_{m \times m} &&&&\\
    & I_{m\times m}&&&\\
     && \ddots && \\
    &&&& I_{m \times m}
    \end{array}
 \right),
 \]
 $\mbox{a $m(n+1) \times m(n+1)$ matrix}$, and
 \[
 P_0=\left(
    \begin{array}{ccccc}
    A_{m\times m} &&&&\\
    & A_{m\times m}&&&\\
    && \ddots&&\\
    &&&& A_{m\times m}
    \end{array}
 \right),
 \]
 $\mbox{a $m(n+1) \times m(n+1)$ matrix}.$
 Thus \eqref{f:full_discrete_01} gives a full discrete system of equation of the form
 \begin{equation}\label{fLfulldisc001}
 \mathcal{W}\underline C = \underline F.
 \end{equation}
 \eqref{fLfulldisc001} needs to be solved for $m(n+1)$ unknowns $C$ after imposing boundary conditions.

\begin{thrm}
The matrix $\mathcal{W}$ is invertible if at least one $a_j(t)\ne 0$, $j=1, 2, \cdots, m$.
\end{thrm}
\begin{proof}
Here $\mathcal{W}$ is defined from \eqref{fLfulldisc001} which is a sum of two matrices
$$\mathcal{A} = P_0 \Psi\ \text{ and }\   \mathcal{B} = \frac{4}{b-a}P_1 \Psi \mathcal{\mathbf{M}}.$$
Now $\mathcal{A}$  is an invertible matrix  if at least one $a_j(t)\ne 0$ $j=1, 2, \cdots, m$, where
$$\Psi = diag(\Psi(t_0),\ \Psi(t_1),\ \cdots,\ \Psi(t_n)).$$  So the matrix $\mathcal{W}$ is nonsingular
 even if $B$ is singular, since  $$\det(A+B)=\det(A) + \det(B).$$
\end{proof}

Now we discuss boundary conditions. We have $m$ boundary conditions. Depending on the physical conditions modelled the $k^{th}$ left and right boundary conditions \eqref{f:eq002_bc_a} or \eqref{f:eq002_bc_b} (or BCs of any type) can be presented by
 \[
    \sum_{j=0}^n C_{k,j} T_j^*(t_n)=\alpha_k,\ \text{and } \ \sum_{j=0}^n C_{k,j} T_j^*(t_0)=\beta_k,
 \]
which can be written in vector form as
 \begin{equation}\label{f:syseq02}
     T^*(t_n) C_{k} = \alpha_k,\ \text{and } \ T^*(t_0) C_{k} = \beta_k.
 \end{equation}
So we need to solve the system of equations \eqref{fLfulldisc001} and \eqref{f:syseq02} for $C_{i, j}$. Now \eqref{fLfulldisc001} and \eqref{f:syseq02} involve the Tchebychev  weights $C_{i, j}$ for the unknown functions $y_i$ at $t=t_0$ and $t=t_n$.
We need to apply  boundary conditions  \eqref{f:syseq02} in \eqref{fLfulldisc001} before we attempt to solve them for $C_{i, j}$.  Here we combine \eqref{fLfulldisc001} and \eqref{f:syseq02} to solve the system for the weights.
 %Here for better understanding we recall the arrangement of the matrix $\mathcal{W}$ which is a $m(n+1)\times m(n+1)$.  Here
\begin{description}
\item[Right Boundary Conditions]  The first $m$ rows of $\mathcal{W}$, $F$ and $\Psi$  have been computed at $t=t_{0}$.  Let us assign them as submatrices $W_1$, $F_1$, and  $\Psi_1$.
Now the $k^{th}$ row of $W_1$, and $F_1$ are assigned for the unknown function $y_k(t)$,  $k=1,\ 2,\ \cdots, m$ at $t=t_0$.  Thus we replace the $k^{th}$ row of $W_1$ by the $k^{th}$ row of $\Psi_1$ ($W_1(k,:) = \Psi_1(k, :)$), and respective $k^{th}$ element of $F_1$ by the boundary value $\beta_k$ ($F_1(k,1)=\beta_k$) for all boundary conditions at $t_0=b$.

 \item[Left Boundary Conditions] The last $m$ rows of $\mathcal{W}$, $F$ and $\Psi$  have been computed at $t=t_{n}$.  Let us assign them as submatrices $W_n$, $F_n$, and  $\Psi_n$.
Now the $k^{th}$ row of $W_n$, and $F_n$ are assigned for the unknown function $y_k(t)$,  $k=1,\ 2,\ \cdots, m$ at $t=t_n$.  Thus we replace the $k^{th}$ row of $W_n$ by the $k^{th}$ row of $\Psi_n$  ($W_n(k,:) = \Psi_n(k,:)$), and respective $k^{th}$ element of $F_n$ by the boundary value $\alpha_k$ ($F_n(k,1)=\alpha_k$) for all boundary conditions at $t_n=a$.
 \end{description}
 %$$
 %  \Psi(t_0) \underline C =\underline G_1,\  \text{and}\ \Psi(t_N) \underline C =\underline G_2.
 %$$
 Thus replacing  the first $m$ rows of  $\mathcal{W}$ by the rearranged $W_1$, the last $m$ rows of  $\mathcal{W}$ by the rearranged $W_n$, first $m$ elements of $F$ by the rearranged $F_1$, and the last $m$ elements of $F$ by the rearranged $F_n$ we get the system of equations of the form
\begin{equation}\label{fLfulldisc002}
 \tilde{\mathcal{W}} \underline C = \underline {\tilde F},
 \end{equation}
 where
 \[
 \tilde{\mathcal{W}} = \left[
                       \begin{array}{c}
                       W_1\\
                       \mathcal{W}((m+1):mn, 1:m(n+1))\\
                       W_n
                       \end{array}
                      \right], \quad
                      \text{and }\quad
 \underline{F} = \left[
                       \begin{array}{c}
                       F_1\\
                       \mathcal{W}((m+1):mn, 1)\\
                       F_n
                       \end{array}
                      \right].
 \]
 The $m(n+1)\times m(n+1)$ system of equations \eqref{fLfulldisc002} can be solved for $C$ by using any standard linear system solver.
 Here
 \begin{itemize}
 \item $C_{1, j}$, $j=0,\ 1,\ 2, \cdots, n$  are the Tchebychev weights for  the solution $y(t)$ of  \eqref{f:eq001},
\item  $C_{2, j}$, $j=0,\ 1,\ 2, \cdots, n$  are the Tchebychev weights for $y'(t)$ of  \eqref{f:eq001},
\item $C_{3, j}$, $j=0,\ 1,\ 2, \cdots, n$  are the Tchebychev weights for $y''(t)$ of  \eqref{f:eq001}, \\
    $\vdots$
\item $C_{m, j}$, $j=0,\ 1,\ 2, \cdots, n$  are the Tchebychev weights for $y^{(m-1)}(t)$ of  \eqref{f:eq001}.
\end{itemize}
     Thus we get a simple algorithm to solve the higher order linear differential equation \eqref{f:eq001} with any set of boundary conditions.  In the next section we inspect some examples to show the efficiency of the scheme.

\section{Numerical experiments and discussions}\label{sec05}
In this section we  present numerical solutions of some higher order BVPs  using the method outlined in the previous section.  First define the set of clustered nodes \eqref{nodes001:f}. We use the transformation  \eqref{trnsfrm01:f} to convert the higher order BVPs to a system of first order BVPs.  Here we consider some examples (BVPs) with exact solutions so that we can present accuracy of our scheme umerically.
\begin{example}
We consider the following BVP
\[
y^{(2)}+y=1, \  y(0)=0,\  y(1) = 1.
\]
The exact solution is  $$y(t) = 1- \cos t + \cot 1  \sin t.$$ The following Figure~\ref{test_01} demonstrates the numerical solution of the $2^{nd}$ order BVP. Here from the error computation we observe that the accuracy of the scheme is of
an exponential order and we need $8$ nodes for a $10$ digit accurate solution.
\begin{figure}[ht!]
     \begin{center}
       \includegraphics[width=0.49\textwidth,height=6.5cm]{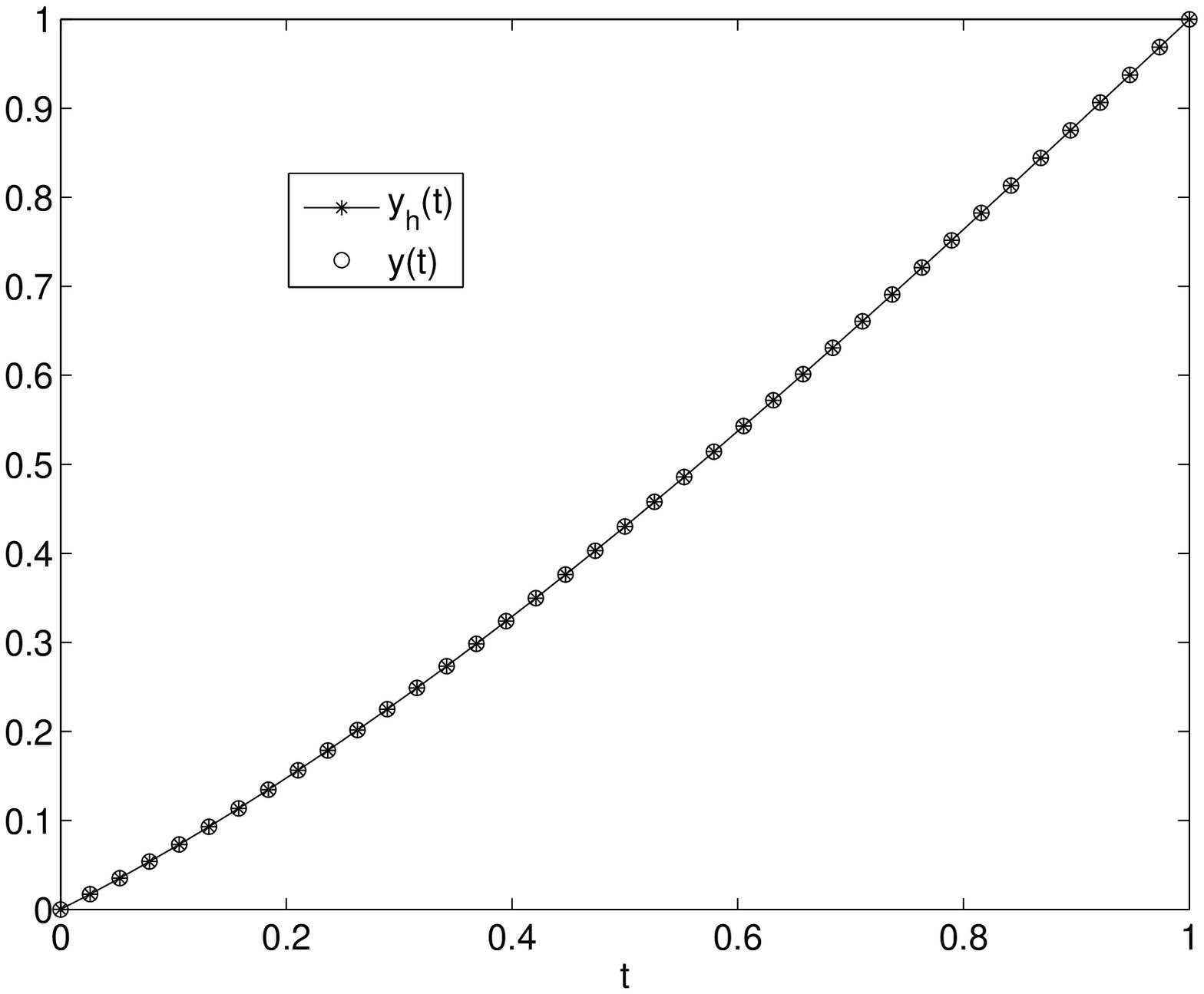}
       \includegraphics[width=0.49\textwidth,height=6.5cm]{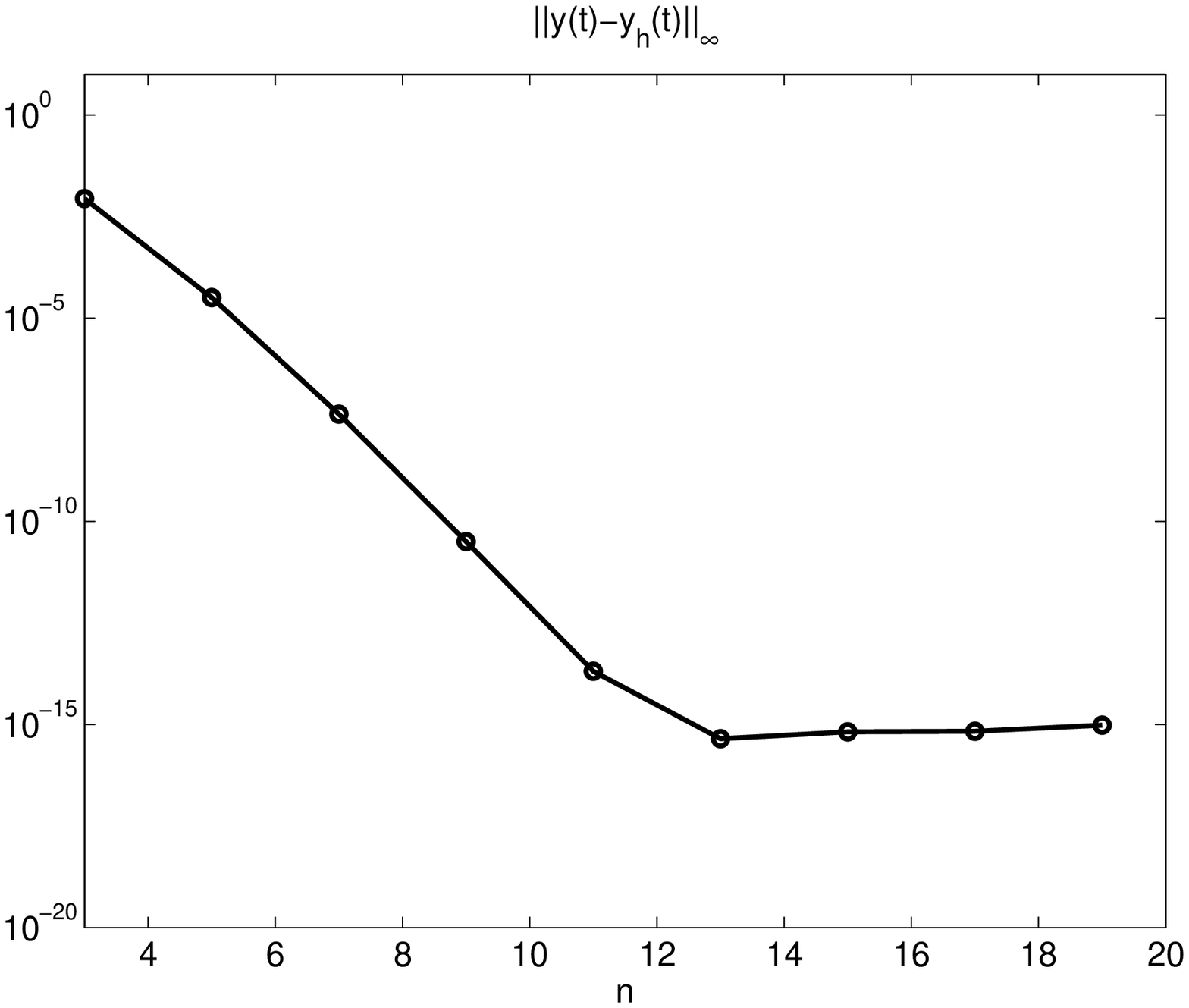}
     \end{center}
\caption{ Solution of the second order BVP. In the left figure we plot the exact and the approximate solutions. In the right figure we plot the computational error.}
    \label{test_01}
    \end{figure}
\end{example}
\begin{example}
We consider the following BVP
\[
y^{(4)} + 2y^{(2)} + y=1, \  y(0)=0,\ y'(0)=0, \  y(1) = 0,\ y'(0)=0.
\]
The exact solution is
\[
y(x) = \frac{1}{1 + sin(1)}(1 - t \cos(1 - t) - \cos(t) + t\cos(t) + sin(1)
       - \sin(1 - t) - \sin(t)).
\]
The following Figure~\ref{test_02} demonstrates the numerical solution of the $4^{th}$ order BVP. Here from the error computation we observe that the accuracy of the scheme is of
an exponential order. Here we notice that we need $9$ nodes for a $10$ digit accurate solution.
\begin{figure}[ht!]
     \begin{center}
       \includegraphics[width=0.49\textwidth,height=6.5cm]{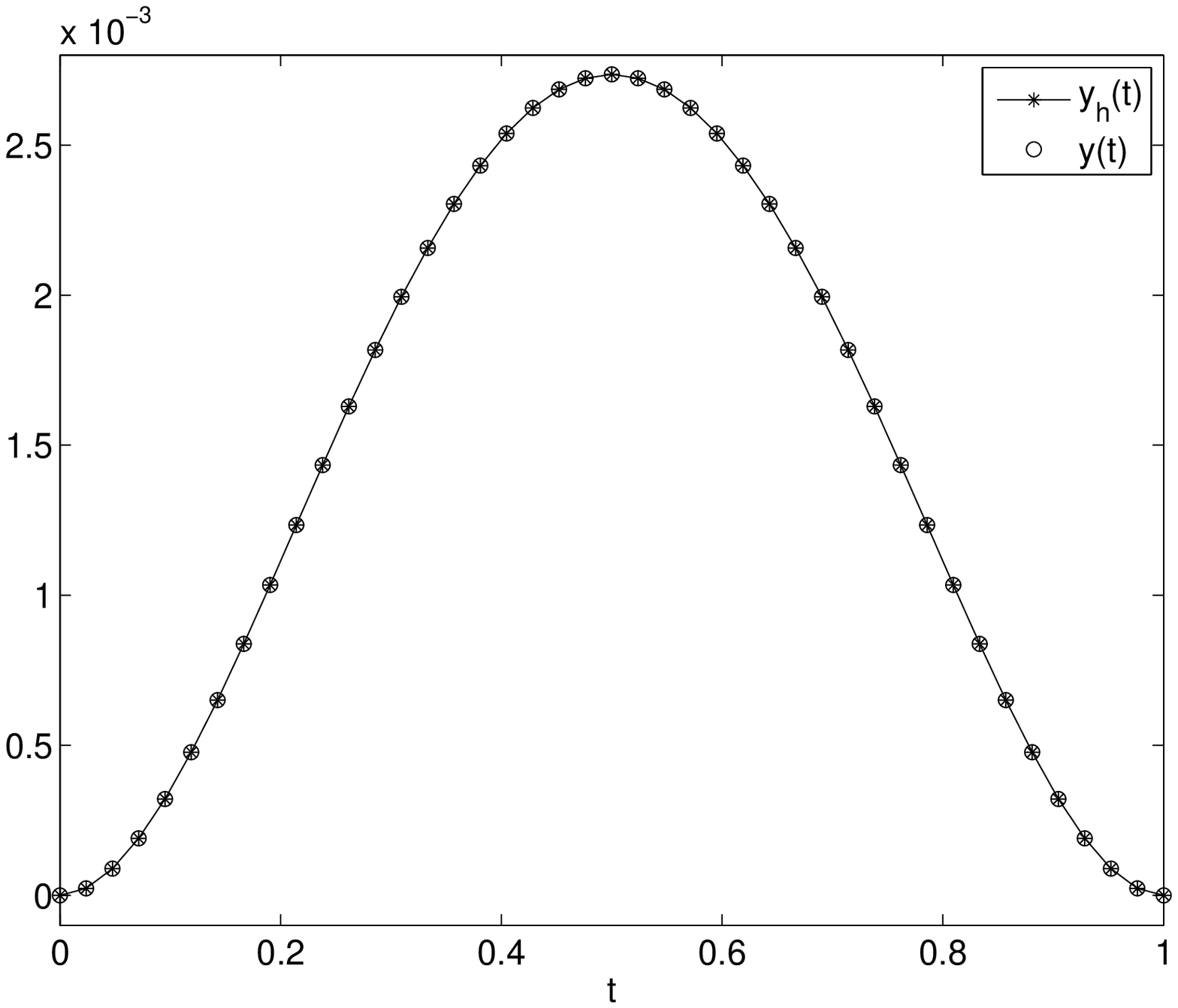}
       \includegraphics[width=0.49\textwidth,height=6.5cm]{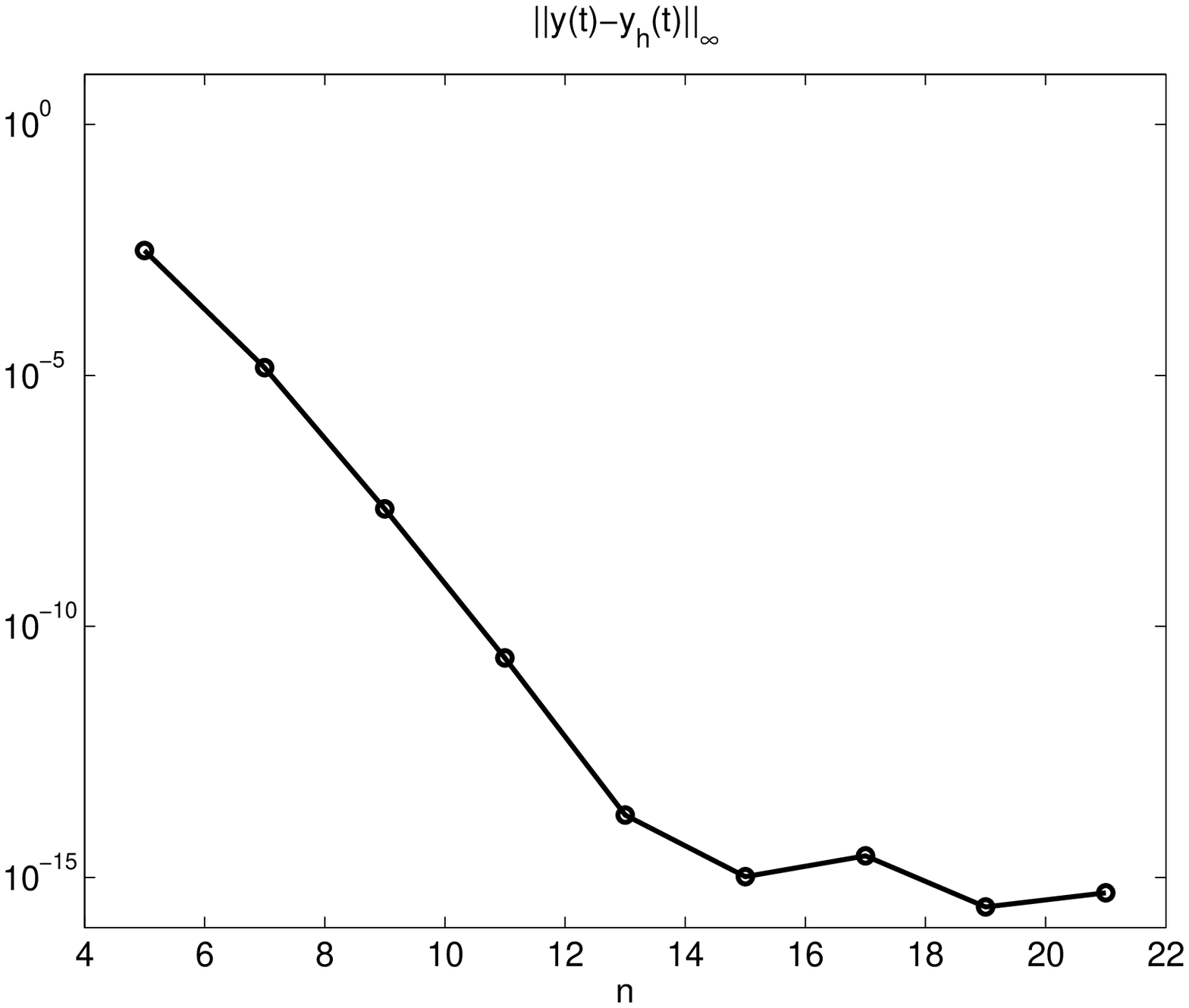}
     \end{center}
\caption{ Solution of the fourth order BVP. In the left figure we plot the exact and the approximate solutions. In the right figure we plot the computational error.}
    \label{test_02}
    \end{figure}
\end{example}
\begin{example}
We consider the following BVP
\[
y^{(5)} -  y = -15 e^x - 10 x e^x, \  y(0)=0,\ y'(0)=1,\ y''(0)=0, \  y(1) = 0,\ y'(1)=-e.
\]
The exact solution is
\[
 y(x) = x(x-1)e^x.
\]
The following Figure~\ref{test_02a} demonstrates the numerical solution of the $5^{th}$ order BVP. Here from the error computation we observe that the accuracy of the scheme is of  an exponential order.  Here we need $11$ nodes for a $10$ digit accurate solution.
\begin{figure}[ht!]
     \begin{center}
       \includegraphics[width=0.49\textwidth,height=6.5cm]{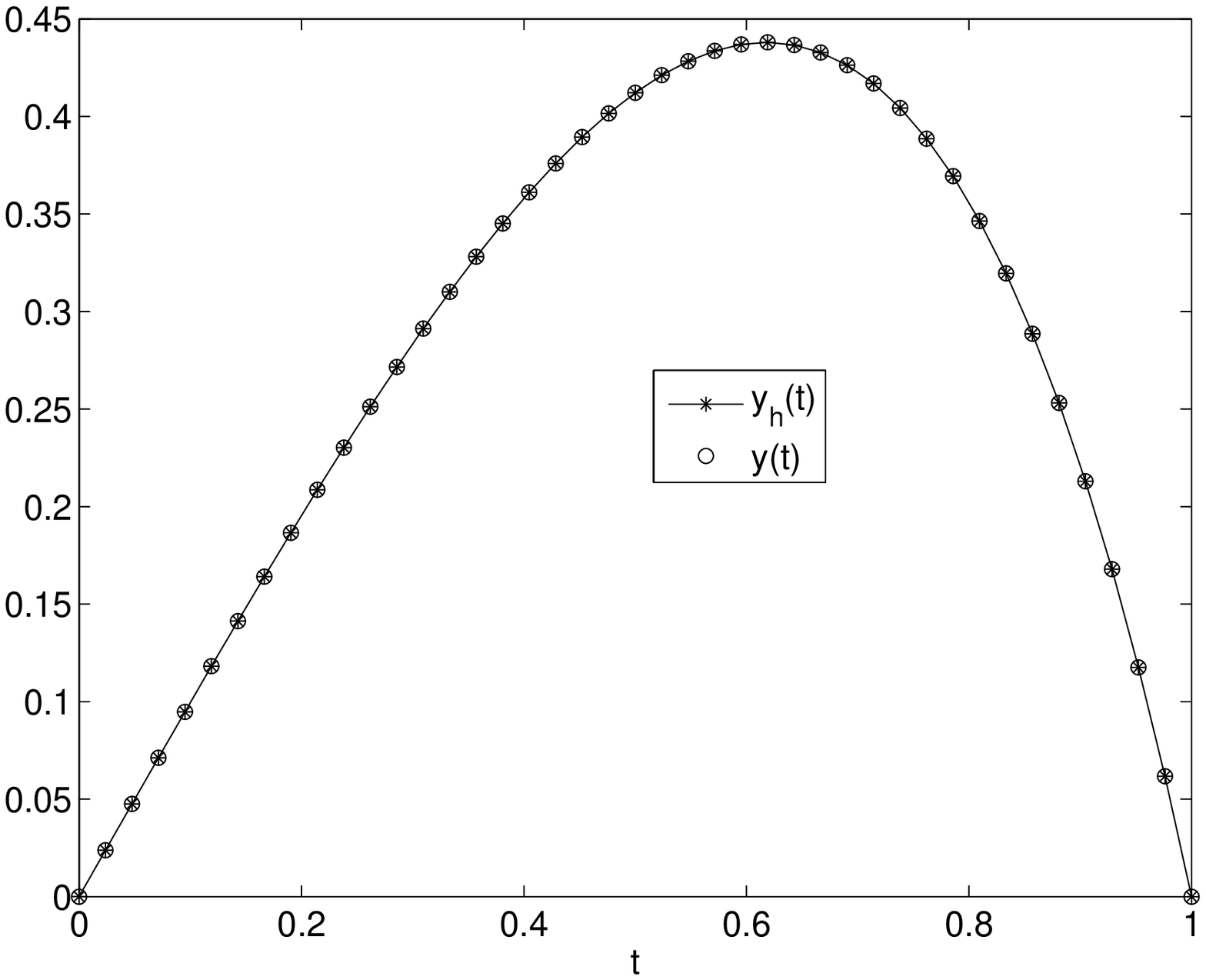}
       \includegraphics[width=0.49\textwidth,height=6.5cm]{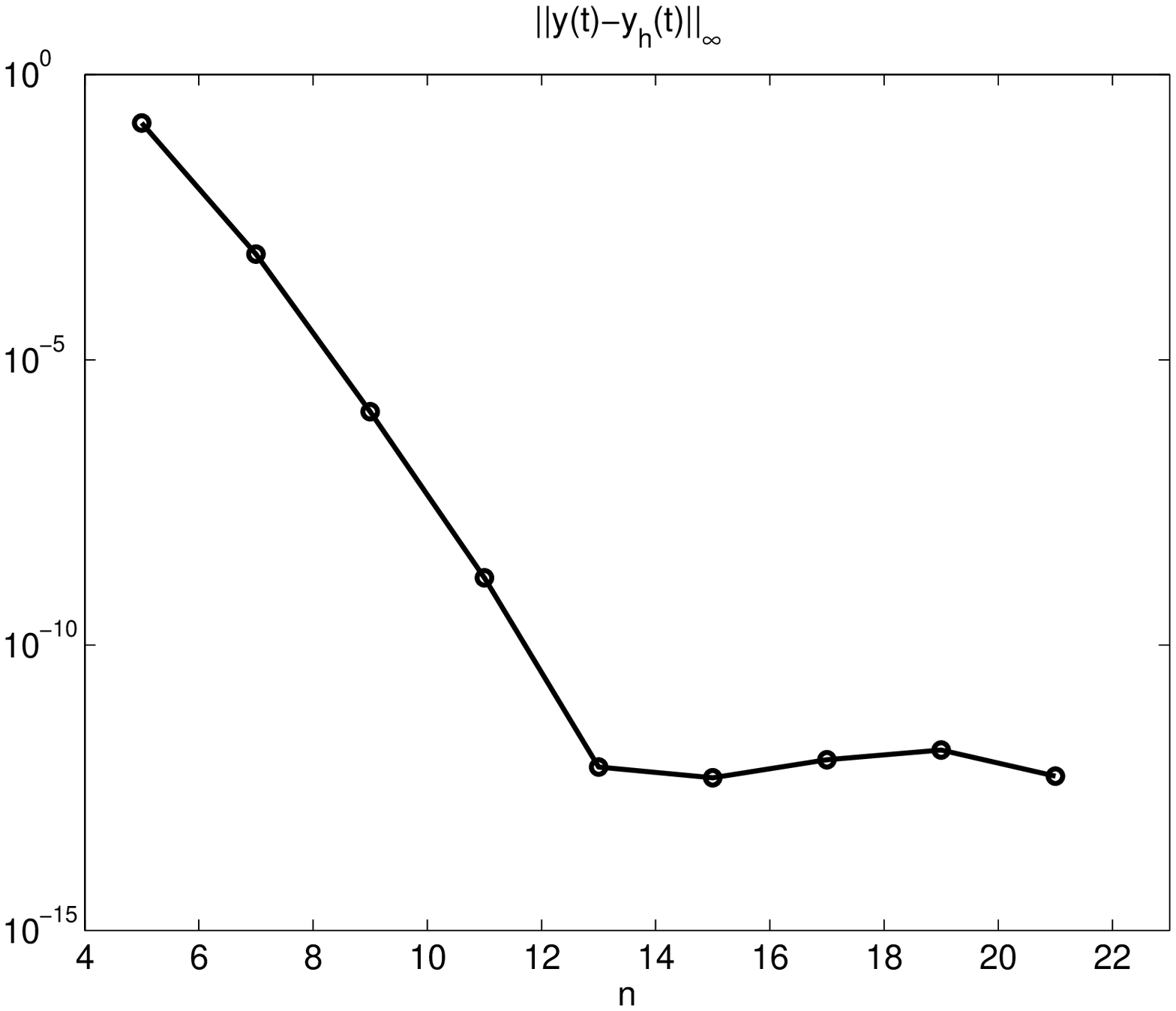}
     \end{center}
\caption{ Solution of the fifth order BVP. In the left figure we plot the exact and the approximate solutions. In the right figure we plot the computational error.}
    \label{test_02a}
    \end{figure}
\end{example}
\begin{example}
Consider the BVP
\[
     y^{(6)} -y = -6 e^x,
\]
with $y(0)=1$, $y'(0)=0$, $y^{(2)}(0)=-1$, $y(1)=0$, $y'(1)=-e$, $y^{(2)}(1)=-2e$.
The exact solution is
\[
  y(x)=(1-x)e^x.
\]
The following Figure~\ref{test_03} demonstrates the numerical solution of the $6^{th}$ order BVP. Here from the error computation we observe that the accuracy of the scheme is of
an exponential order and  $11$ nodes are enough for a $10$ digit accurate solution.
\begin{figure}[ht!]
     \begin{center}
       \includegraphics[width=0.49\textwidth,height=6.5cm]{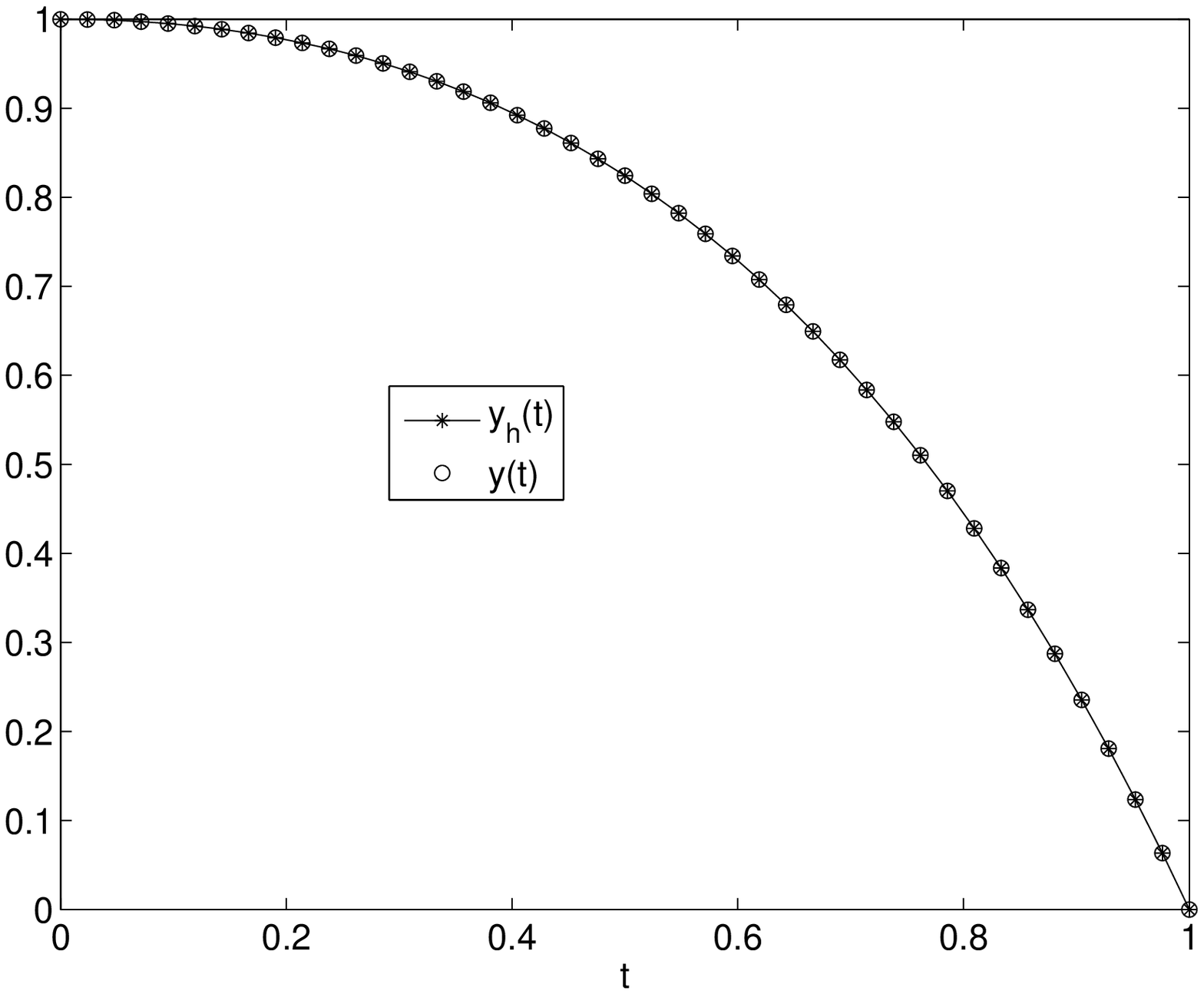}
       \includegraphics[width=0.49\textwidth,height=6.5cm]{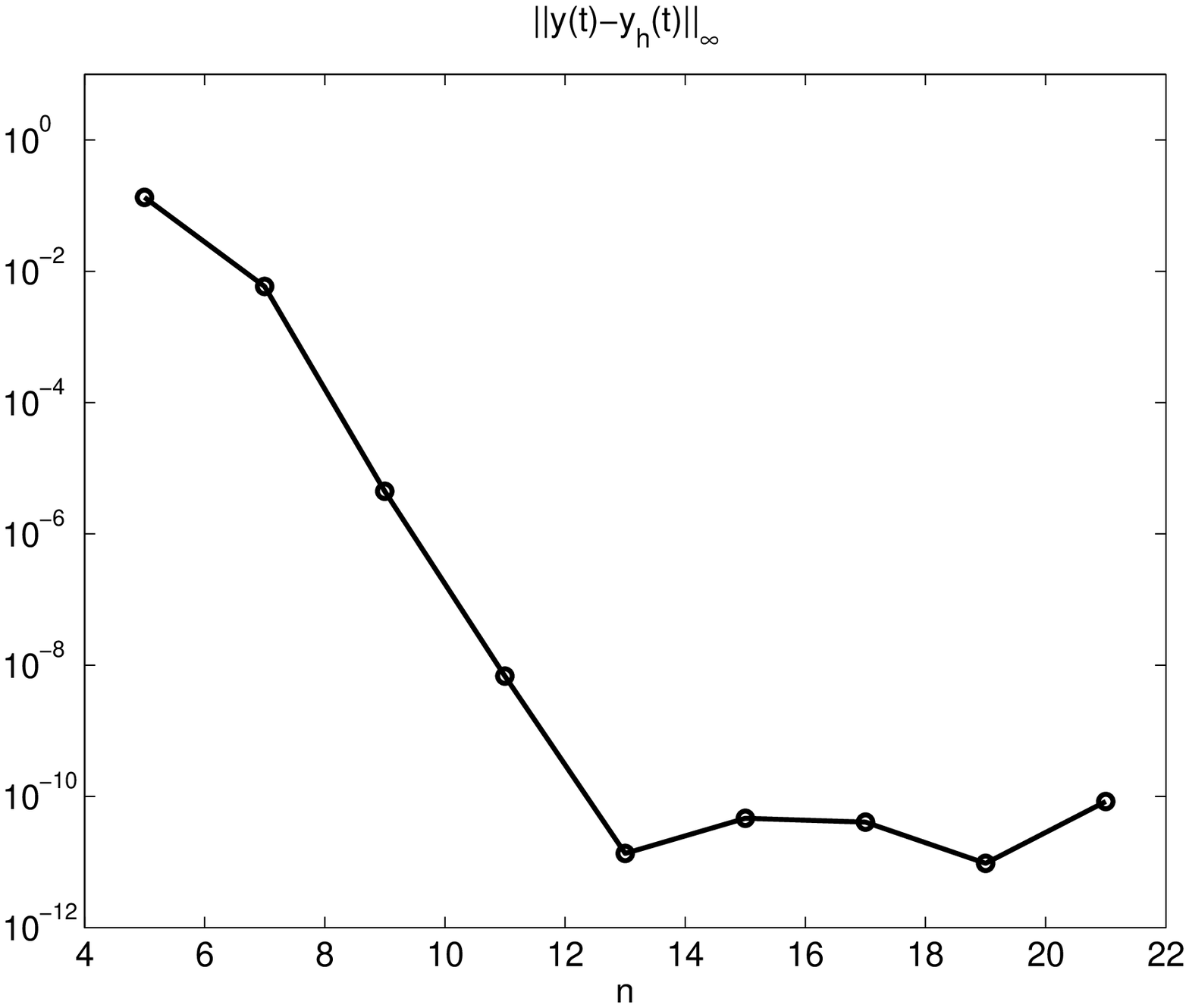}
     \end{center}
\caption{ Solution of the sixth order BVP. In the left figure we plot the exact and the approximate solutions. In the right figure we plot the computational error.}
    \label{test_03}
    \end{figure}
\end{example}
\begin{example}
Consider the BVP
\[
     y^{(9)} -y = -9 e^t,
\]
with $y(0)=1$, $y'(0)=0$, $y^{(2)}(0)=-1$,$y^{(3)}(0)=-2$,$y^{(4)}(0)=-3$, $y(1)=0$, $y'(1)=-e$, $y^{(2)}(1)=-2e$, $y^{(3)}(1)=-3e$.
The exact solution is
\[
  y(t)=(1-x)e^t.
\]
The following Figure~\ref{test_04} demonstrates the numerical solution of the $9^{th}$ order BVP. Here from the error computation we observe that the accuracy of the scheme is of
an exponential order. Here we see that we need $13$ nodes for a $8$ digit accurate solution.
\begin{figure}[ht!]
     \begin{center}
       \includegraphics[width=0.49\textwidth,height=6.5cm]{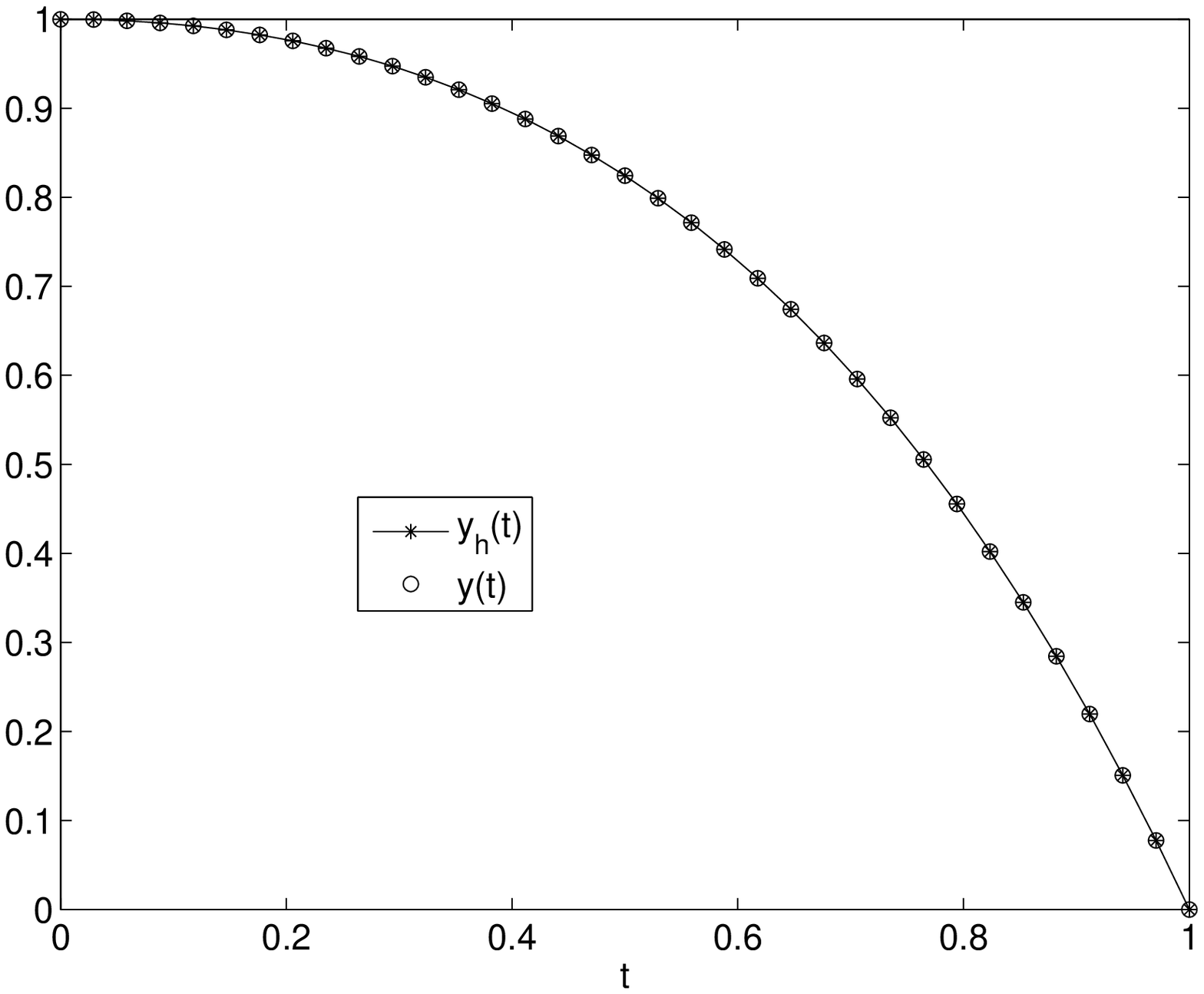}
       \includegraphics[width=0.49\textwidth,height=6.5cm]{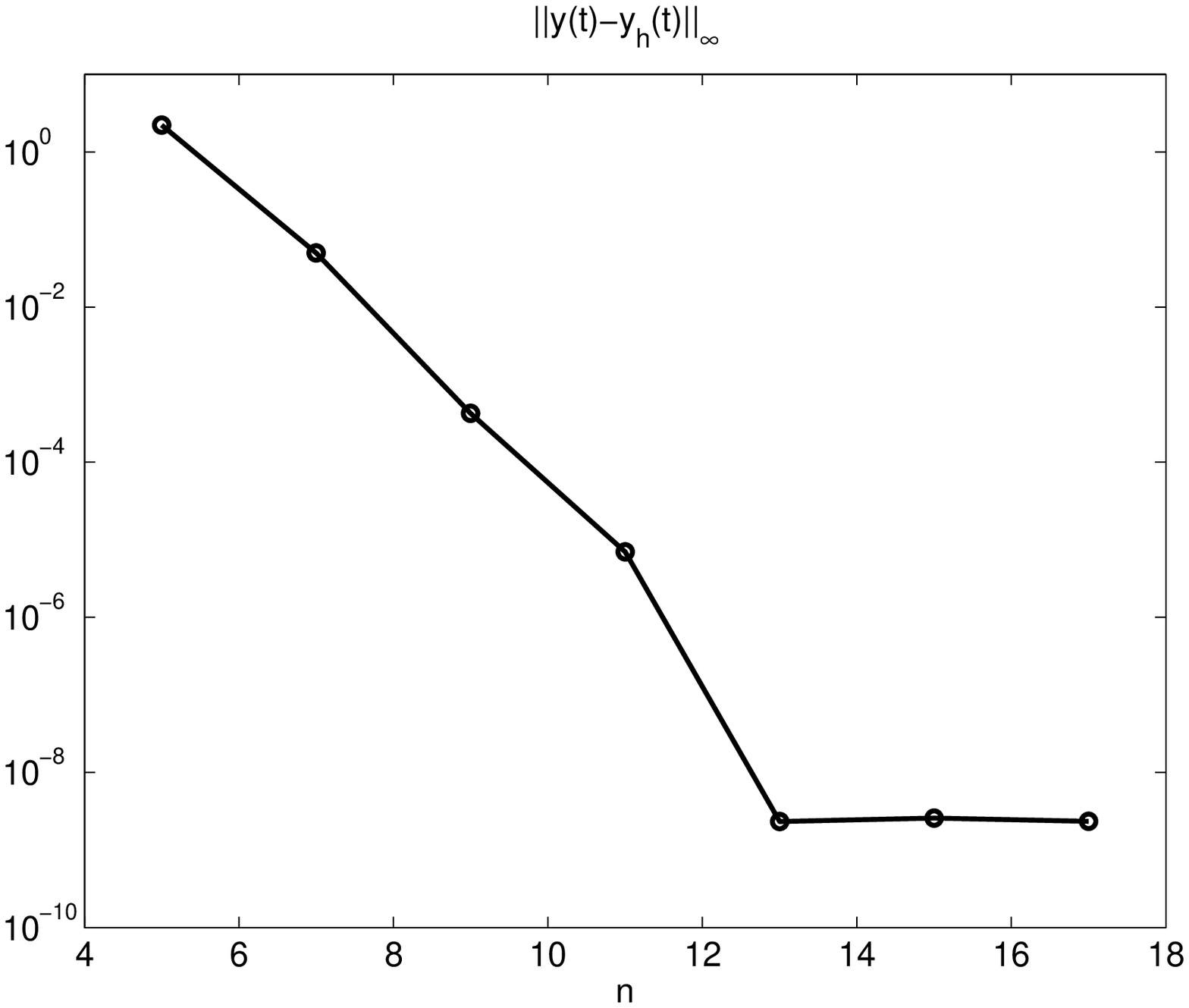}
     \end{center}
\caption{Solution of the ninth order BVP. In the left figure we plot the exact and the approximate solutions. In the right figure we plot the computational error.}
    \label{test_04}
    \end{figure}
\end{example}

%From the examples we notice that a $8$ to $10$ digit accuracy can be obtained by using $12-13$ Tchebychev nodes only.
%
We discuss in detail the formulations of a spectral collocation method for a general $m^{th}$  order BVPs using Tchebychev polynomials as  basis functions. From these computations we notice that the result is minimum $8-15$ digit accurate when $n>10$ which agree with other published works available in the literature.  Here the advantage of using this scheme is that we can handle boundary conditions simply and efficiently.  It preserves spectral accuracy as one expects from polynomial approximations at the Tchebychev notes.
Therefore, we conclude that the Tchebychev polynomial approach for the transformed system of first order BVPs produce much accurate results than all other previously published lower order works. From these computations we see that our method compete very well with other methods cited in this article.

There are some benefits of using the transformed scheme for \eqref{f:eq001}. Here we do not need to compute higher order derivatives of the unknown function. In this presented scheme, by converting the higher order differential equation \eqref{f:eq001} to the first order system of differential equations \eqref{f:eq002aa} we avoid computing Tchebychev differentiation matrices (what one needs while using Tchebychev polynomials to approximate higher order derivatives  for a direct scheme for \eqref{f:eq001}\cite{Trefethen}).  All boundary conditions $y^{(i)}(a)=\alpha_i$ and $y^{(i)}(b)=\beta_i$ have been converted to boundary conditions $y_i(a)=\alpha_i$, and $y_i(b) = \beta_i$ respectively, which are  simple to handle in our proposed set up.
Whereas applying $y^{(i)}(a)$ and $y^{(i)}(b)$ in the Tchebychev differentiation matrices is really complicated (see \cite{Trefethen, WeidemanR2000} for the exact detail). Thus the scheme we present here is simply efficient for solving the higher order BVP \eqref{f:eq001}, and preserves spectral accuracy.
\bibliography{ref}{}
\bibliographystyle{plain}
\end{document}